\documentclass[12pt]{amsart}
\usepackage{amsfonts,amsmath,amssymb, amscd}
\usepackage[all]{xy}

\newtheorem{theorem}{Theorem}[section]
\newtheorem{definition}[theorem]{Definition}
\newtheorem{proposition}[theorem]{Proposition}
\newtheorem{lemma}[theorem]{Lemma}
\newtheorem{corollary}[theorem]{Corollary}
\newtheorem{question}[theorem]{Question}

\theoremstyle{remark}
\newtheorem{remark}[theorem]{Remark}

\begin{document}

\title{Matrix models for noncommutative algebraic manifolds}

\author{Teodor Banica}
\address{T.B.: Department of Mathematics, Cergy-Pontoise University, 95000 Cergy-Pontoise, France.}
\email{{\tt teodor.banica@u-cergy.fr}}

\author{Julien Bichon}
\address{J.B.: Laboratoire de Math\'ematiques, Universit\'e Blaise Pascal, Campus universitaire des C\'ezeaux, 3 place Vasarely,
63178 Aubi\`ere Cedex, France.} \email{{\tt Julien.Bichon@math.univ-bpclermont.fr}}

\subjclass[2010]{46L89, 20G42, 16S38}
\keywords{Noncommutative algebraic manifold, Matrix model, Quantum group}

\begin{abstract}
We discuss the notion of matrix model, $\pi:C(X)\to M_K(C(T))$, for algebraic submanifolds of the free complex sphere, $X\subset S^{N-1}_{\mathbb C,+}$. When $K\in\mathbb N$ is fixed there is a universal such model, which factorizes as $\pi:C(X)\to C(X^{(K)})\subset M_K(C(T))$. We have $X^{(1)}=X_{class}$ and, under a mild assumption, inclusions $X^{(1)}\subset X^{(2)}\subset X^{(3)}\subset\ldots\subset X$. Our main results concern $X^{(2)},X^{(3)},X^{(4)},\ldots$, their relation with various half-classical versions of $X$, and lead
to the construction of families of higher half-liberations of the complex spheres and of the unitary groups, all having faithful matrix models. 
\end{abstract}

\maketitle

\section*{Introduction}

There are several possible definitions for the noncommutative algebraic manifolds. According to a well-known theorem of Gelfand, one reasonable point of view is that the noncommutative analogues of the compact real algebraic manifolds $X_{class}\subset\mathbb C^N$ should be the abstract spectra of the universal $C^*$-algebras of the following type:
$$C(X)=C^*\left(z_1,\ldots,z_N\Big|P_i(z_1,\ldots,z_N, z_1^*, \ldots ,z_N^*)=0\right)$$

Here the family of noncommutative polynomials $\{P_i\}$ must be such that the maximal $C^*$-norm on the universal $*$-algebra $<z_1,\ldots,z_N|P_i(z_1,\ldots,z_N,z_1^*, \ldots ,z_N^*)=0>$ is bounded. In order to avoid this issue, we will restrict here attention to the algebraic submanifolds of the free complex sphere, $X\subset S^{N-1}_{\mathbb C,+}$. That is, we will assume that the polynomial relations $P_i(z_1,\ldots,z_N)=0$ defining $X$ include the following two relations:
$$\sum_iz_iz_i^*=\sum_iz_i^*z_i=1$$

Associated to $X$ is its classical version, obtained as Gelfand spectrum of the algebra $C(X_{class})=C(X)/I$, where $I\subset C(X)$ is the commutator ideal. We have:
$$X_{class}=\left\{z\in\mathbb C^N\Big|P_i(z_1,\ldots,z_N, \overline{z_1}, \ldots ,\overline{z_N})=0\right\}$$

The general liberation philosophy is that of viewing $X$ as a ``liberation'' of $X_{class}$. This point of view was intensively developed in the quantum group case, starting with Wang's papers \cite{wa1}, \cite{wa2}. Several extensions, to the case of noncommutative homogeneous spaces, or more general manifolds, have been developed recently \cite{ba1}, \cite{ba2}, \cite{bi1}, \cite{bdu}.

We will be interested here in an alternative point of view, more analytical, coming from random matrix theory. Generally speaking, a matrix model for a noncommutative manifold $X$ is a representation of $C^*$-algebras, as follows:
$$\pi:C(X)\to M_K(C(T))$$

Here $T$ is a compact space, and $K<\infty$. This is of course the general algebraic framework. Further axioms can include the fact that $T$ is a compact Lie group, or an homogeneous space, or an abstract compact probability space. Observe that, with this latter assumption, $M_K(C(T))$ is a usual random matrix space, in the sense of probability theory, and we can obtain an integration functional on $X$, simply by setting:
$$\int_X\varphi=\frac{1}{K}\sum_{i=1}^K\int_T\pi(\varphi)_{ii}$$

In the quantum group case, there is a whole machinery devoted to the study of such models.
Our purpose here is to start an adaptation work for these methods, to the algebraic manifold case. We will extend one of the simplest available technologies, namely the $2\times2$ matrix model picture of the  ``half-liberation'' procedure \cite{bsp}, \cite{bve}, discussed in \cite{ba1}, \cite{bi1}, \cite{bdu}.

In order to explain our results, let us go back to the general matrix models for the algebraic manifolds, $\pi:C(X)\to M_K(C(T))$. When $K\in\mathbb N$ is fixed, one can abstractly construct a ``maximal'' such model, and this model must factorize as follows:
$$\pi:C(X)\to C(X^{(K)})\subset M_K(C(T))$$

Here $X^{(K)}\subset X$ is the closed subspace obtained by taking the image of $\pi$. Under a mild assumption,  we obtain in this way an algebraic submanifold of $X$, and  an increasing sequence of algebraic submanifolds of $X$, as follows:
$$X^{(1)}\subset X^{(2)}\subset X^{(3)}\subset\ldots\ldots\subset X$$
with $X^{(1)}=X_{class}$. In general, $X^{(K)}\subset X$ can be thought of as being the ``part of $X$ which is realizable with $K\times K$ random matrices''.

Our main results will concern the analogues of the equality $X_{class}= X^{(1)}$, for the higher order manifolds $X^{(2)},X^{(3)},X^{(4)},\ldots$ 
Our starting point is that for $X=S^{N-1}_{\mathbb R, *}$, the real half-liberated sphere or $X=  S^{N-1}_{\mathbb C, **}$, the complex half-liberated sphere, we have $X=X^{(2)}$. Investigating the general case  $K \geq 1$  will  lead 
to the construction of an operation $$X\to X_{1/K-class}$$ (with $X \subset S^{N-1}_{\mathbb C, +}$ assumed to be $K$-symmetric, see Section 6) which at $K=1$ is the operation $X\to X_{class}$, and with, at any $K$ 
$$X_{1/K-class}\subset X^{(K)}$$
In particular starting with $X=S^{N-1}_{\mathbb C, +}$, we will obtain the construction of a family 
 $K$-half liberated sphere  $S^{N-1}_{\mathbb C, K}$ with 
$$(S^{N-1}_{\mathbb C, K})_{class}=S^{N-1}_{\mathbb C, 1}=S^{N-1}_{\mathbb C}, \ S^{N-1}_{\mathbb C, 2}=S^{N-1}_{\mathbb C, **}, \ S^{N-1}_{\mathbb C, K}\subset (S^{N-1}_{\mathbb C, +})^{(K)}$$
 and $S^{N-1}_{\mathbb C, K}\not = S^{N-1}_{\mathbb C, K'}$ for $K \not=K'$.

Summarizing, our results produce higher versions of the previously known half-liberated spheres, and bring some non-trivial information on $X^{(K)}$ in general. 

Let us also mention that our framework also includes the  case of compact quantum groups, and produces quantum groups that are new, and could provide some interesting input for the classification program for the ``easy quantum groups'' in \cite{fre}, \cite{rwe},  \cite{tw2}.

As limit cases of the higher half-liberations we construct, we also get a sphere $S^{N-1}_{\mathbb C, \infty}$ and quantum group $U_{N, \infty}$ that we believe to be of interest. 
For the reader who is familiar with quantum group easiness, let us mention that the easy quantum group $U_{N, \infty}$ 
comes from the following diagrams:
$$
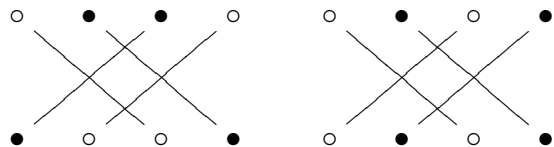
\xymatrix@R=12mm@C=5mm{\circ\ar@{-}[drr]&\bullet\ar@{-}[drr]&\bullet\ar@{-}[dll]&\circ\ar@{-}[dll]\\\bullet&\circ&\circ&\bullet}\quad \quad 
\xymatrix@R=12mm@C=5mm{\circ\ar@{-}[drr]&\bullet\ar@{-}[drr]&\circ\ar@{-}[dll]&\bullet\ar@{-}[dll]\\\circ&\bullet&\circ&\bullet}$$


The paper is organized as follows: in Section 1 we recall the set-up for noncommutative algebraic manifolds. In Section 2 we discuss matrix models and the universal matrix model.  Sections 3-5 are devoted to the construction of higher half liberated spheres together with the construction of the associated faithful matrix models. Section 6 introduces the construction of the $1/K$-classical version of noncommutative manifolds. Section 7 briefly explains how the previous considerations apply as well to compact quantum groups.
In the final Section 8, we define limit versions of our previous spheres and unitary quantum groups.

\medskip

\textbf{Acknowledgements.} We thank Simon Riche for useful discussions on pure tensors, and the referee for pointing out reference \cite{har} on this topic.

\section{Noncommutative algebraic manifolds}


Let us recall that  the Gelfand theorem enables one to reconstruct  a compact space $X$  from $C(X)$, the algebra of continuous functions on
$X$, and conversely states that any commutative $C^*$-algebra (we assume that $C^*$-algebras are unital) is of this form. To be more precise, given a commutative $C^*$-algebra $A$, the underlying compact space $X=Spec(A)$ is the set of characters $\chi:A\to\mathbb C$, with topology making the evaluation maps continuous.

In view of Gelfand's theorem, we have the following traditional definition:

\begin{definition}
The category of noncommutative compact spaces is the category of unital $C^*$-algebras, with the arrows reversed. Given a noncommutative compact space $X$, coming from a $C^*$-algebra $A$, we write $A=C(X)$ and $X=Spec(A)$, and call $X$ the abstract spectrum of $A$.
\end{definition}

Observe that the category of usual compact spaces embeds into the category of noncommutative compact spaces. More precisely, a compact space $X$ corresponds to the noncommutative space associated to the algebra $A=C(X)$. In addition, in this situation, $X$ can be recovered as a Gelfand spectrum, $X=Spec(A)$. 

In this framework, an inclusion of $Y \subset X$ of noncommutative spaces corresponds to a surjective $C^*$-algebra
map $C(X) \rightarrow C(Y)$. Any noncommutative compact space $X$ contains a maximal classical compact subspace:

\begin{definition}
The classical version $X_{class}\subset X$ of a noncommutative compact space $X$ is defined by
$$C(X_{class})= C(X)_{ab}$$
where $C(X)_{ab}$ is  the quotient of $C(X)$ by the commutator ideal.

\end{definition}

As an illustration, let us discuss the case of the noncommutative algebraic manifolds. As yet another consequence of the Gelfand theorem, we can formulate:

\begin{definition}\label{def:algv0}
The noncommutative analogues of the compact real algebraic manifolds $X\subset\mathbb R^N$, $Y\subset\mathbb C^N$ are the abstract spectra of the universal $C^*$-algebras of type
\begin{eqnarray*}
C(X)&=&C^*\left(z_1,\ldots,z_N\Big|z_i=z_i^*,P_i(z_1,\ldots,z_N)=0\right)\\
C(Y)&=&C^*\left(z_1,\ldots,z_N\Big|P_i(z_1,\ldots,z_N, z_1^*, \ldots , z_N^*)=0\right)
\end{eqnarray*}
where the family of noncommutative polynomials $\{P_i\}$ is such that the maximal $C^*$-norm on the universal $*$-algebras on the right is bounded.
\end{definition}

This is of course an abstract definition, with the boundeness condition on the maximal $C^*$-norm being a real issue. We will discuss this issue in what follows.

In the context of Definition \ref{def:algv0}, the classical versions of $X,Y$, are given by 
\begin{eqnarray*}
X_{class}&=&\left\{z\in\mathbb R^N\Big|P_i(z_1,\ldots,z_N)=0\right\}\\
Y_{class}&=&\left\{z\in\mathbb C^N\Big|P_i(z_1,\ldots,z_N, \overline{z_1}, \ldots , \overline{z_N})=0\right\}
\end{eqnarray*}
Conversely, any such manifolds $X_{class},Y_{class}$ can be obtained from Definition \ref{def:algv0}, by adding the commutation relations between $z_i,z_i^*$ to the defining relations $P_i=0$.

Let us go back now to the boundedness condition in Definition 1.3. This is a true technical issue, and in order to avoid it, and to work with a much lighter formalism, we will assume that our manifolds appear as submanifolds of the ``free spheres''. This is also supported by the fact that a compact topological manifold can always be realized as closed subspace of an Euclidean  sphere.  

Consider the standard sphere, $S^{N-1}_\mathbb R=\{z\in\mathbb R^N|\sum_iz_i^2=1\}$, and the standard complex sphere, $S^{N-1}_\mathbb C=\{z\in\mathbb C^N|\sum_i|z_i|^2=1\}$. In order to discuss the free analogues of these spheres, we must first understand the associated algebras $C(S^{N-1}_\mathbb R),C(S^{N-1}_\mathbb C)$. The  well-known result here, coming from the Gelfand theorem, is as follows:

\begin{proposition}
We have the presentation results
$$C(S^{N-1}_\mathbb R)=C^*_{comm}\left(z_1,\ldots,z_N\Big|z_i=z_i^*,\sum_iz_i^2=1\right)$$
$$C(S^{N-1}_\mathbb C)=C^*_{comm}\left(z_1,\ldots,z_N\Big|\sum_iz_iz_i^*=\sum_iz_i^*z_i=1\right)$$
where by $C^*_{comm}$ we mean universal commutative $C^*$-algebra.
\end{proposition}

We can now proceed with ``liberation'', as follows \cite{bago, ba0}:

\begin{definition}
Associated to any $N\in\mathbb N$ are the universal $C^*$-algebras
\begin{eqnarray*}
C(S^{N-1}_{\mathbb R,+})&=&C^*\left(z_1,\ldots,z_N\Big|z_i=z_i^*,\sum_iz_i^2=1\right)\\
C(S^{N-1}_{\mathbb C,+})&=&C^*\left(z_1,\ldots,z_N\Big|\sum_iz_iz_i^*=1=\sum_iz_i^*z_i\right)
\end{eqnarray*}
whose abstract spectra $S^{N-1}_{\mathbb R,+},S^{N-1}_{\mathbb C,+}$ are called free analogues of $S^{N-1}_\mathbb R,S^{N-1}_\mathbb C$.
\end{definition}

Observe that the above two algebras are indeed well-defined, because the relations show that we have $||z_i||\leq1$, for any $C^*$-norm. Thus the biggest $C^*$-norm is bounded, and the above two enveloping $C^*$-algebras are well-defined.


We can now introduce the manifolds that we are interested in:

\begin{definition}
A closed subspace $X\subset S^{N-1}_{\mathbb C,+}$ is called algebraic when
$$C(X)=C(S^{N-1}_{\mathbb C,+})\Big/\Big<P_i(z_1,\ldots,z_N, z_1^*, \ldots , z_N^*)=0,\forall i\in I\Big>$$
for a certain family of noncommutative polynomials $P_i\in\mathbb C<z_1,\ldots,z_N, z_1^*, \ldots z_N^*>$.
\end{definition}

Observe that $S^{N-1}_\mathbb R,S^{N-1}_\mathbb C,S^{N-1}_{\mathbb R,+},S^{N-1}_{\mathbb C,+}$ are algebraic manifolds, and  there are many other examples.

Given $X \subset S^{N-1}_{\mathbb C,+}$, we denote by $\mathcal O(X)$ the $*$-subalgebra of $C(X)$ generated by the elements $z_i$ (the coordinate algebra of $X$), and requiring that $X$ is algebraic precisely means that $C(X)$ is the enveloping $C^*$-algebra of $\mathcal O(X)$.

In order to present some interesting classes of examples, we recall from Wang's paper \cite{wa1} that the free analogues of $O_N,U_N$ are constructed as follows:
\begin{eqnarray*}
C(O_N^+)&=&C^*\left((u_{ij})_{i,j=1,\ldots,N}\Big|u=\bar{u},u^t=u^{-1}\right)\\
C(U_N^+)&=&C^*\left((u_{ij})_{i,j=1,\ldots,N}\Big|u^*=u^{-1},u^t=\bar{u}^{-1}\right)
\end{eqnarray*}

To be more precise, $O_N^+,U_N^+$ are compact matrix quantum groups in the sense of Woronowicz \cite{wo1}, \cite{wo2}, with comultiplication, counit and antipode as follows:
$$\Delta(u_{ij})=\sum_ku_{ik}\otimes u_{kj}\quad,\quad\varepsilon(u_{ij})=\delta_{ij}\quad,\quad S(u_{ij})=u_{ji}^*$$

We recall that a closed quantum subgroup $G\subset U_N^+$, with standard coordinates denoted $v_{ij}$, is called full when $C(G)$ is the enveloping $C^*$-algebra of the $*$-algebra generated by the variables $v_{ij}$. As a basic example, the discrete group algebras $C^*(\Gamma)$ are full, while the reduced algebras $C^*_{red}(\Gamma)$, with $\Gamma$ not amenable, are not full. See \cite{ntu}, \cite{wo1}.

With this convention, we have the following result:

\begin{theorem}
The following are noncommutative algebraic manifolds:
\begin{enumerate}
\item The real algebraic submanifolds $X\subset S^{N-1}_\mathbb C$.

\item The finite quantum subspaces $X\subset S^{N-1}_{\mathbb C,+}$.

\item The full closed quantum subgroups $G\subset U_N^+$.
\end{enumerate}
\end{theorem}

\begin{proof}
All these results are well-known, the proof being as follows:

(1) This is clear from definitions. Observe also that, conversely, a closed subset $X\subset S^{N-1}_{\mathbb C}$ is algebraic precisely when it is a real algebraic manifold, in the usual sense.

(2) When the subspace $X\subset S^{N-1}_{\mathbb C,+}$ is finite, in the sense that the algebra $C(X)$ is finite dimensional, we have $C(X)=\mathcal O(X)$.

(3) Our claim here is that we have inclusions of algebraic manifolds, as follows:
$$G\subset U_N^+\subset S^{N^2-1}_{\mathbb C,+}$$

Indeed, regarding the inclusion at right, let $u_{ij}$ be the standard coordinates of $U_N^+$. Since $u=(u_{ij})$ is biunitary we have $\sum_ju_{ij}u_{ij}^*=\sum_ju_{ij}^*u_{ij}=1$ for any $i$, so the rescaled variables $z_{ij}=u_{ij}/\sqrt{N}$ satisfy the equations for $S^{N^2-1}_{\mathbb C,+}$. In addition, since the biunitarity conditions on $u$ are algebraic, we obtain in this way an algebraic submanifold:
$$C(U_N^+)=C(S^{N^2-1}_{\mathbb C,+})\Big/\Big(zz^*=z^*z=\bar{z}z^t=z^t\bar{z}=\frac{1}{N}\cdot1_N\Big)$$

Regarding the inclusion at left, this comes by definition, and what is left to prove is that $G\subset U_N^+$ is algebraic. But this follows from Woronowicz's Tannakian results in \cite{wo2}. Indeed, in the orthogonal case, $G\subset O_N^+$, we have the following presentation result:
$$C(G)=C(O_N^+)\Big/\Big(T\in Hom_G(u^{\otimes k},u^{\otimes l}),\forall k,l\in\mathbb N\Big)$$
where $Hom_G(u^{\otimes k},u^{\otimes l})$ denotes the space of morphisms of representations, and the notation means that each $T \in 
Hom_G(u^{\otimes k},u^{\otimes l})$ defines a family of algebraic relations between the $u_{ij}'s$, in the standard way.

In the unitary case the proof is similar, replacing the tensor powers of $u$ by tensor powers of $u$ and $\bar{u}$.  See \cite{mal}.
\end{proof}



\section{Matrix models}

We discuss now the notion of matrix model. 
For $X\subset S^{N-1}_{\mathbb C,+}$ infinite, there is no faithful representation of $C(X)$ into a matrix algebra $M_K(\mathbb C)$, and we will use the following notion.

\begin{definition}\label{def:mod}
A matrix model for $X\subset S^{N-1}_{\mathbb C,+}$ is a morphism of $C^*$-algebras
$$\pi:C(X)\to M_K(C(T))$$
where $T$ is a compact space, and $K\geq1$ is an integer.
\end{definition}

As already mentioned in the introduction, this is of course just the general framework, and $T$ might have some more structure.
In this paper we will focus on the basic theory, and use Definition \ref{def:mod} as it is.

As a first example, at $K=1$ a matrix model is simply a morphism of $C^*$-algebras $\pi:C(X)\to C(T)$, with $T$ being a compact space. Such a morphism must come from a continuous map $p:T\to X_{class} \subset X$, and if $\pi$ is assumed to be faithful, then $X=X_{class}$ and $p$ must be surjective. 

To generalize the above considerations at $K\geq 2$, we will use the following definition.

\begin{definition}
Let  $X\subset S^{N-1}_{\mathbb C,+}$. We define $X^{(K)} \subset X$ by
$$C(X^{(K)}) = C(X)/J_K$$
where $J_K$ is the intersection of the kernels of all matrix representations $C(X) \rightarrow M_L(\mathbb C)$, for any $ L \leq K$.
\end{definition}

Clearly the definition can be made for any $C^*$-algebra. We have
$$X_{class}=X^{(1)} \subset X^{(2)} \subset X^{(3)} \ldots \ldots \subset  X$$
and $X^{(\infty)}=\bigcup_{K\geq1}X^{(K)} = X$ if and only $C(X)$ is residually finite-dimensional, see \cite{chi} for a recent paper on that topic, in the context of quantum groups.

\begin{proposition}\label{prop:X=XK}
 Let $Y \subset X\subset S^{N-1}_{\mathbb C,+}$. Then $Y \subset X^{(K)}$ if and only if any irreducible representation of $C(Y)$ has dimension $\leq K$. In particular $X^{(K)} = X$ if and only if any irreducible representation of $C(X)$ has dimension $\leq K$.
\end{proposition}

\begin{proof}
If any irreducible representation of $C(Y)$ has dimension $\leq K$, then $Y \subset X^{(K)}$ follows from the standard fact that the irreducible representations of a $C ^*$-algebra separate its points, see e.g. \cite{dix}. Conversely, if $Y\subset X^{(K)}$, it is enough to show that  any  irreducible representation of $C(X^{(K)})$ has dimension $\leq K$: this follows from a polynomial identity argument, as in \cite[Proposition 3.6.3]{dix}.
\end{proof}

The connection with the previous considerations is:

\begin{proposition}\label{prop:faithimplX=XK}
If $X\subset S^{N-1}_{\mathbb C,+}$ has a faithful matrix model
$C(X)\to M_K(C(T))$, then $X=X^{(K)}$.
\end{proposition}

\begin{proof}
This follows from Proposition \ref{prop:X=XK} and 
standard representation theory \cite{dix}: the irreducible representations of $M_K(C(T))$ all have dimension $K$, and an irreducible representation of a subalgebra is always isomorphic to a  subrepresentation of an irreducible representation of the big algebra.
\end{proof}

We now discuss the universal $K\times K$-matrix model, 
 a $C^*$-algebra analogue of character varieties for discrete groups or finite-dimensional algebras, see e.g. \cite{lubmag,kra}.

\begin{proposition}\label{prop:univmodel}
 Given  $X\subset S^{N-1}_{\mathbb C,+}$ algebraic, the category of its $K\times K$ matrix models, with $K\geq 1$ being fixed, has a universal object
$\pi_K:C(X)\to M_K(C(T_K))$.
This means that if $\rho : C(X) \rightarrow M_K(C(T))$ is any matrix model, there exists
a commutative diagram
$$\xymatrix{C(X)\ar[rr]^{\pi}\ar[rd]_{\rho}&&M_K(C(T_K))\ar@{-->}[dl]\\&M_K(C(T))&}$$
where the vertical map on the right is unique and arises from a continuous map $T\to T_K$.
\end{proposition}

\begin{proof}
Consider the universal commutative $C^*$-algebra generated by elements $x_{ij}(a)$, with $1 \leq i,j \leq K$, $a \in \mathcal O(X)$, subject to the relations ($a,b \in \mathcal O(X)$, $\lambda \in \mathbb C$, $1 \leq i,j\leq K$):
$$x_{ij}(a+\lambda b) = x_{ij}(a)+\lambda x_{ij}(b), \ x_{ij}(ab)= \sum_kx_{ik}(a)x_{kj}(b)$$
$$x_{ij}(1)=\delta_{ij}, \ x_{ij}(a)^*= x_{ji}(a^*)$$  
This indeed well-defined because of the relations $\sum_l \sum_k x_{ik}(z_l^*)x_{ki}(z_l)=1$. Let $T_K$ be the spectrum of this $C^*$-algebra. Since $X$ is algebraic, we get a matrix model
$$\pi : C(X) \rightarrow M_K(C(T_K)), \ \pi(z_k) = (x_{ij}(z_k))$$
and it is immediate, by construction of $T_K$ and $\pi$, that we have the announced universal matrix model.
\end{proof}

\begin{proposition}\label{prop:kerunivmodel}
 Let $X\subset S^{N-1}_{\mathbb C,+}$ with $X$ algebraic and $X_{class} \not= \emptyset$. Let $\pi:C(X)\to M_K(C(T_K))$ be the universal matrix model.
Then we have 
$$C(X^{(K)})= C(X)/Ker(\pi)$$
and hence $X=X^{(K)}$ if and only if $X$ has a faithful $K \times K$-matrix model.
\end{proposition}

\begin{proof}
 We have to show that $Ker(\pi)=J_K$, the latter ideal being the intersection of the kernels of all matrix representations $C(X) \rightarrow M_L(\mathbb C)$, for any $ L \leq K$. For $a \not \in Ker(\pi)$, we see that $a \not \in J_K$ by evaluating at an appropriate element of $T_K$.

Conversely, let $a \in Ker(\pi)$. Let $\rho : C(X) \rightarrow M_L(\mathbb C)$ be a representation with $L \leq K$, and let $\varepsilon : C(X) \rightarrow \mathbb C$ be a representation. We extend $\rho$ to a representation $\rho' : C(X) \rightarrow M_K(\mathbb C)$ by letting, for any $b \in C(X)$,
$$\rho'(b) = \begin{pmatrix} \rho(b) & 0 \\
              0 & \varepsilon(b) I_{K-L}  
   \end{pmatrix}$$
and the universal property of the universal matrix model yields that $\rho'(a)=0$, since $\pi(a)=0$. Hence $\rho(a)=0$. 
We thus have $a \in J_K$, and $Ker(\pi)\subset J_K$, and the first statement is proved.
The last statement follows from the first one and Proposition \ref{prop:faithimplX=XK}
\end{proof}

\begin{proposition}\label{prop:univmodelalgebraic}
 Let $X\subset S^{N-1}_{\mathbb C,+}$ with $X$ algebraic and $X_{class} \not=\emptyset$. Then $X^{(K)}$ is algebraic as well.
\end{proposition}

\begin{proof}
 We retain the notation in the proof of Proposition \ref{prop:univmodel}, and consider the map
$\pi_0 : \mathcal O(X) \rightarrow M_K(C(T_K))$, $z_l \mapsto (x_{ij}(z_l))$. It induces a $*$-algebra map 
$$\tilde{\pi_0} : C^*(\mathcal O(X)/Ker(\pi_0)) \rightarrow M_K(C(T_K))$$
We need to show that $\tilde{\pi}_0$ is injective. Indeed, since the universal model factorizes 
$$\pi : C(X) \overset{p}\to C^*(\mathcal O(X)/Ker(\pi_0))\overset{\tilde{\pi}_0}\to M_K(C(T_K))$$
where $p$ is canonical surjection, we will get that $Ker(\pi)= Ker(p)$, and hence, according to the previous proposition, that $C(X^{(K)})= C(X)/Ker(p) =  C^*(\mathcal O(X)/Ker(\pi_0))$, showing that $X^{(K)}$ is indeed algebraic. 

Since $\mathcal O(X)/Ker(\pi_0)$ is isomorphic to a $*$-subalgebra of $M_K(C(T_K))$, it satisfies the standard Amitsur-Levitski polynomial identity $S_{2K}(x_1, \ldots , x_{2K})=0$, and by density so does  $C^*(\mathcal O(X)/Ker(\pi_0))$. Hence any irreducible representation of  $C^*(\mathcal O(X)/Ker(\pi_0))$ has dimension $\leq K$ (again see the proof of Proposition 3.6.3 in \cite{dix}). Thus if $a \in  C^*(\mathcal O(X)/Ker(\pi_0))$ is a nonzero element, we can, by the same reasoning as in the proof of the previous proposition, find a representation $\rho :  C^*(\mathcal O(X)/Ker(\pi_0)) \rightarrow M_K(\mathbb C)$ such that $\rho(a)\not=0$ (because a given algebra map $\varepsilon : C(X) \rightarrow \mathbb C$ induces an algebra map $C(T_K) \rightarrow \mathbb C$, $x_{ij}(a) \mapsto \delta_{ij} \varepsilon(a)$, which enables us to extend representations similarly as before). By construction the universal model space yields an algebra  map $M_K(C(T_K)) \rightarrow M_K(\mathbb C)$ whose composition with $\tilde{\pi_0}p=\pi$ is $\rho p$, so $\tilde{\pi_0}(a)\not=0$, and $\tilde{\pi}_0$ is injective.
\end{proof}

Summarizing the results of the section, we have proved:

\begin{theorem}
  Let $X\subset S^{N-1}_{\mathbb C,+}$ with $X$ algebraic and $X_{class} \not=\emptyset$.
Then we have an increasing sequence of algebraic submanifolds
$$X_{class}= X^{(1)}\subset X^{(2)}\subset X^{(3)}\subset\ldots\ldots\subset X$$
where $C(X^{(K)}) \subset M_K(C(T_K))$ is obtained by factorizing the universal matrix model. 
\end{theorem}

\section{Higher versions of half-liberated complex spheres}

In this section we  define, for any $K\geq 2$, a $K$-half-liberated sphere, and study its first basic properties.
As a warm-up, let us recall the definitions  of the various half-liberated spheres.

\begin{definition}
The noncommutative spaces $S^{N-1}_{\mathbb R,*} \subset S^{N-1}_{\mathbb C,**}\subset S^{N-1}_{\mathbb C,*}\subset S^{N-1}_{\mathbb C,+}$ defined by
\begin{eqnarray*}
C(S^{N-1}_{\mathbb R,*})&=&C(S^{N-1}_{\mathbb R,+})\Big/\Big<abc=cba,\forall a,b,c\in\{z_i\}\Big>\\
C(S^{N-1}_{\mathbb C,**})&=&C(S^{N-1}_{\mathbb C,+})\Big/\Big<abc=cba,\forall a,b,c\in\{z_i,z_i^*\}\Big>\\
C(S^{N-1}_{\mathbb C,*})&=&C(S^{N-1}_{\mathbb C,+})\Big/\Big<ab^*c=cb^*a,\forall a,b,c\in\{z_i\}\Big>
\end{eqnarray*}
are called respectively the half-liberated real sphere, the half-liberated complex sphere and the full half-liberated complex sphere.
\end{definition}

These spheres, which are obviously algebraic, arose as natural quantum homogeneous spaces over appropriate quantum groups.
\begin{enumerate}
 \item $S^{N-1}_{\mathbb R,*}$ corresponds to the half-liberated orthogonal quantum group $O_N^*$ from \cite{bsp}. It is known that $C(S^{N-1}_{\mathbb R,*})$ has a faithful $2\times 2$ matrix model \cite{bi1}, so that  $(S^{N-1}_{\mathbb R,*})^{(2)}=S^{N-1}_{\mathbb R,*}$, and the subspaces $X \subset S^{N-1}_{\mathbb R,*}$ are well understood.
\item  $S^{N-1}_{\mathbb C,**}$ corresponds to some half-liberated unitary quantum group $U_N^{**}$ from \cite{bdu}.
We have $S^{N-1}_{\mathbb C,**} \subset S^{2N-1}_{\mathbb R, *}$, and $C(S^{N-1}_{\mathbb C,**})$ has a faithful $2\times 2$ matrix model as well.
\item $S^{N-1}_{\mathbb C,*}$ corresponds to the full half-liberated unitary quantum group $U_N^*$ from \cite{bdd}, and is more mysterious.
\end{enumerate}

The following result will be our starting point  to define ``higher'' versions of $S^{N-1}_{\mathbb C,**}$.

\begin{proposition}\label{prop:incluhalf}
Let $X\subset S^{N-1}_{\mathbb C,+}$, with coordinates $z_1,\ldots,z_N$.
\begin{enumerate}
\item $X\subset S^{N-1}_{\mathbb C,*}$ precisely when $\{z_iz_j^*\}$ commute, and $\{z_i^*z_j\}$ commute as well.

\item $X\subset S^{N-1}_{\mathbb C,**}$ precisely when the variables $\{z_iz_j,z_iz_j^*,z_i^*z_j,z_i^*z_j^*\}$ all commute.
\end{enumerate}
\end{proposition}

\begin{proof}
Regarding the first assertion, the implication ``$\implies$'' follows from the following two computations, using the $ab^*c=cb^*a$ rule:
$$ab^*cd^*=cb^*ad^*=cd^*ab^*$$ 
$$a^*bc^*d=c^*ba^*d=c^*da^*b$$

As for the implication ``$\Longleftarrow$'', this is obtained as follows, by using the commutation assumptions in the statement, and by summing over $e=z_i$:
$$ae^*eb^*c=ab^*ce^*e=ce^*ab^*e=cb^*ee^*a\implies ab^*c=cb^*a$$

The proof of the second assertion is similar, because we can remove all the $*$ signs, except for those concerning $e^*$, and use the above computations with $a,b,c,d\in\{z_i,z_i^*\}$.
\end{proof}

We  now define, for any $K\geq 2$, a $K$-half-liberated sphere.

\begin{definition}
For $K\geq 2$,  the noncommutative space $ S^{N-1}_{\mathbb C,K}\subset S^{N-1}_{\mathbb C,+}$ defined by
$$C(S^{N-1}_{\mathbb C, K})= C(S^{N-1}_{\mathbb C,+})\Big/\Big< [z_{i_1}\cdots z_{i_K}, z_{j_{1}} \cdots z_{j_K}]=0
=  [z_{i_1} \cdots z_{i_K}, z_{j_1}^* \cdots z_{j_K}^*]\Big>$$
is called the $K$-half-liberated complex sphere.
\end{definition}

It is clear that these spheres are algebraic.
The definition also makes  sense at $K=1$, with $S^{N-1}_{\mathbb C,1}=S^{N-1}_{\mathbb C}$. As for the $2$-half-liberated complex sphere, it indeed coincides with the half-liberated complex sphere from the previous section.

\begin{proposition}\label{prop:Khalf}
 The following relations hold in $C(S_{\mathbb C, K}^{N-1})$:
\begin{enumerate}
\item $z_{i_1}(z_{i_2} \cdots z_{i_K})z_{i_{K+1}} =  z_{i_{K+1}}(z_{i_2} \cdots z_{i_K})z_{i_1}$,
\item $[z_iz_j^*,z_kz_l^*]=[z_i^*z_j,z_k^*z_l]=0 =[z_iz_j^*,z_k^*z_l]$,
\item $[z_{i_1}\cdots z_{i_K}, z_iz_j^*]=0= [z_{i_1}\cdots z_{i_K}, z_i^*z_j].$
\end{enumerate}
In particular we have $S_{\mathbb C, K}^{N-1} \subset S^{N-1}_{\mathbb C,*}$, and at $K=2$ we have 
$S_{\mathbb C, 2}^{N-1} = S^{N-1}_{\mathbb C,**}$
\end{proposition}

\begin{proof}
Let $A$ be the $C^*$-subalgebra of $C(S_{\mathbb C, K}^{N-1})$ generated by the elements of the form $z_{i_1}\cdots z_{i_K}$. By construction $A$ is a commutative $C^*$-algebra, since it is generated by elements that pairwise commute. We have 
$$z_iz_j^* = \sum_{\alpha_1, \ldots , \alpha_{K-1}} z_i z_{\alpha_1} \cdots z_{\alpha_{K-1}} z_{\alpha_{K-1}}^* \cdots z_{\alpha_{1}}^*z_j^*$$
and hence $z_iz_j^* \in A$. Similarly $z_i^*z_j \in A$.
 Hence the commutativity of $A$ ensures that the elements $z_{i_1}\cdots z_{i_K}$, $z_i^*z_j$ and $z_iz_j^*$ all commute, and this gives the second and third relations. We get 
\begin{align*}z_{i_1}(z_{i_2} \cdots z_{i_K})z_{i_{K+1}} & =\sum_j z_{i_1}z_{i_2} \cdots z_{i_K}z_jz_j^*z_{i_{K+1}}
=\sum_jz_{i_1}z_j^*z_{i_{K+1}}z_{i_2} \cdots z_{i_K}z_j \\
&= \sum_j z_{i_{K+1}}z_{i_2} \cdots z_{i_K}z_{i_1}z_j^*z_j = z_{i_{K+1}}(z_{i_2} \cdots z_{i_K})z_{i_1}
\end{align*}
which gives the first relations. The last assertion then follows from Proposition \ref{prop:incluhalf}.
\end{proof}

We will see in Section 5 that even more commutation relations hold in $C(S_{\mathbb C, K}^{N-1})$.

\begin{remark}\label{rem:nonreal}
It would of course be possible to define a real version by  $$C(S^{N-1}_{\mathbb R, K})= C(S^{N-1}_{\mathbb R,+})\Big/\Big< [z_{i_1}\cdots z_{i_K}, z_{j_{1}} \cdots z_{j_K}]=0\Big>$$
For $K$ even, Propositions \ref{prop:incluhalf} and \ref{prop:Khalf} give that $C(S^{N-1}_{\mathbb R, K})=C(S^{N-1}_{\mathbb R, *})$, while for $K$ odd it is not difficult to see that $C(S^{N-1}_{\mathbb R, K})=C(S^{N-1}_{\mathbb R})$. Hence nothing new is obtained in the real case. In the complex case, we will see in Corollary \ref{cor:irrSNK} that for $K \not = K'$, the $C^*$-algebras $C(S^{N-1}_{\mathbb C, K})$ and $C(S^{N-1}_{\mathbb C, K'})$ are not isomorphic.
\end{remark}

Our goal now is to construct a faithful matrix model for $S_{\mathbb C, K}^{N-1}$. We will use the following construction. Consider $(S^{N-1}_{\mathbb C})^K$, the product of $K$ copies of $S^{N-1}_{\mathbb C}$, that we endow
with  the action of the cyclic group $\mathbb Z_K=\langle \tau\rangle$ given by cyclic permutation of the factors:
$$\tau(z_0, \ldots z_{K-1}) = (z_{K-1}, z_0, \ldots, z_{K-2})$$
Denote  by $a_{i,c}$, $1\leq i \leq N$, $0 \leq c \leq K-1$,  the canonical generators of  $C((S^{N-1}_{\mathbb C})^K)$, with $a_{ic}(z_0,  \cdots, z_{K-1}) = (z_{c})_i$  (the second indices are considered modulo $K$).

The above  action of $\mathbb Z_K$ induces a $C^*$-action on $C((S^{N-1}_{\mathbb C})^K)$, $\tau (a_{i,c})=a_{i,c+1}$. We thus form the crossed product $C^*$-algebra $ C((S^{N-1}_{\mathbb C})^K) \rtimes \mathbb Z_K$.

\begin{proposition}\label{prop:mapcrossed}
 We have a $*$-algebra map
\begin{align*}
\pi : C(S^{N-1}_{\mathbb C, K}) &\longrightarrow C((S^{N-1}_{\mathbb C})^K) \rtimes \mathbb Z_K  \\
z_i &\longmapsto a_{i,0} \otimes \tau
\end{align*}
\end{proposition}

\begin{proof}
 The existence of $\pi$ follows from the verification that the elements  $a_{i,0} \otimes \tau$ satisfy the defining relations of $C(S^{N-1}_{\mathbb C, K})$. We have
$$\sum_i(a_{i,0} \otimes \tau)(a_{i,0} \otimes \tau)^* = \sum_i (a_{i,0} \otimes \tau)(a_{i,K-1}^* \otimes \tau^{-1})= \sum_i a_{i,0} a_{i,0}^*\otimes 1=1 \otimes 1$$
and similarly
$$\sum_i(a_{i,0} \otimes \tau)^*(a_{i,0} \otimes \tau) = 1 \otimes 1$$
We also have 
$$a_{i_1,0} \otimes \tau \cdots a_{i_{K},0} \otimes \tau = a_{i_1,0} a_{i_2,1} \dots a_{i_{K},K-1}\otimes 1$$
and $$(a_{i_1,0} \otimes \tau)^* \cdots (a_{i_{K},0} \otimes \tau)^* = a_{i_1,K-1}^* a_{i_2,K-2}^* \cdots a_{i_{K},0}^*\otimes 1$$
We conclude easily from these identities.
\end{proof}

To prove the injectivity of the above map $\pi$, we will need some auxiliary material, developed in the next section.

\section{Pure tensors}

In this section we establish some technical results, for later use, in order to prove the injectivity of the map in Proposition \ref{prop:mapcrossed}. 

Let $V,W$ be finite dimensional vector spaces. Recall that an element $X\in V\otimes W$ is said to be a pure tensor if $X=v\otimes w$ with $v\in V\setminus \{0\},w\in W\setminus \{0\}$. We denote by $\mathcal P(V\otimes W)$ the set of pure tensors. 
It is immediate that $\mathcal P(V\otimes W)$ can be identified with the Segre variety $\Sigma_{V,W}$ \cite{har}, that is the image of the Segre map
\begin{align*}
\sigma :  \mathbb P(V) \times \mathbb P(W) &\longrightarrow\mathbb P(V \otimes W) \\
([v], [w]) &\longmapsto [v \otimes w] 
\end{align*}
More generally now, if $V_1,\ldots,V_K$ are finite dimensional vector spaces, we say that $X\in V_1\otimes\cdots\otimes V_K$ is a pure tensor if $X=v_1\otimes\cdots\otimes v_K$, with $v_1\in V_1\setminus \{0\},\ldots,v_K \in V_K\setminus \{0\}$, and we denote by $\mathcal P(V_1\otimes\cdots\otimes V_K)$ the set of pure tensors.

Working now in $(\mathbb C^N)^{\otimes K}$ endowed with its canonical basis, the following result characterizes the pure tensors:

\begin{lemma}\label{lemm:pureKcoor} 
 Let $r \in (\mathbb C^N)^{\otimes K}$ with 
$$r= \sum_{i_1, \ldots , i_K} r_{i_1, \ldots , i_K} e_{i_1} \otimes \cdots \otimes e_{i_K}$$
Then $r \in \mathcal P((\mathbb C^N)^{\otimes K})$ if and only if 
$$r_{i_1, i_2, \ldots ,i_K} r_{j_1, j_2, \ldots ,j_K} = r_{l_1, l_2 ,\ldots ,l_K} r_{k_1, k_2 \ldots ,k_K}$$
whenever $\{i_1,j_1\}=\{l_1,k_1\}, \{i_2,j_2\}=\{l_2,k_2\}, \ldots , \{i_K,j_K\}= \{l_K, j_K\}$. 
\end{lemma}

\begin{proof}
 At $K=2$, the given equations are those that define the  Segre variety as an algebraic variety, see e.g. the top of page 26 in \cite{har}. The result at $K>2$ is easily shown by induction.
\end{proof}

We now endow $\mathbb C^N$ with its canonical Hilbert space structure.

\begin{proposition}\label{prop:prespure}
 The set $\mathcal P_u((\mathbb C^N)^{\otimes K})$ of pure tensors in $(\mathbb C^N)^{\otimes K}$ of norm $1$ is a compact subspace of $(\mathbb C^N)^{\otimes K}$, and $C(\mathcal P_u((\mathbb C^N)^{\otimes K}))$ is isomorphic to the universal commutative $C^*$-algebra with generators $r_{i_1, \ldots , i_K}$,  with $i_1, \ldots , i_K \in \{1, \ldots ,  N\}$, subject to the relations 
$$\sum_{i_1, \ldots , i_K} r_{i_1, i_2, \ldots ,i_K} r_{i_1, i_2, \ldots ,i_K}^*=1$$
$$r_{i_1, i_2, \ldots ,i_K} r_{j_1, j_2, \ldots ,j_K} = r_{l_1, l_2 ,\ldots ,l_K} r_{k_1, k_2 \ldots ,k_K}$$ 
whenever $\{i_1,j_1\}=\{l_1,k_1\}, \{i_2,j_2\}=\{l_2,k_2\}, \cdots , \{i_K,j_K\}= \{l_K, j_K\}$. 
\end{proposition}

\begin{proof}
 It follows from Lemma \ref{lemm:pureKcoor} that $\mathcal P_u((\mathbb C^N)^{\otimes K})$ is closed in $(\mathbb C^N)^{\otimes K}$, and hence $\mathcal P_u((\mathbb C^N)^{\otimes K})$ is a closed and bounded subset of $(\mathbb C^N)^{\otimes K}$, so is compact. Now let $r= \sum_{i_1, \ldots , i_K} r_{i_1, \ldots , i_K} e_{i_1} \otimes \cdots \otimes e_{i_K}$ satisfying the relations in the statement. Then $r$ is a pure tensor by Lemma \ref{lemm:pureKcoor}, and 
$$||r||^2= \sum_{i_1, \ldots , i_K} r_{i_1, i_2, \ldots ,i_K} \overline{r_{i_1, i_2, \ldots ,i_K}}=1$$
Thus $r$ is pure tensor of norm $1$, and the result follows from Gelfand duality.
\end{proof}

We now link pure tensors and spheres. Let $\mathbb T^{K-1}$ be the subgroup of $\mathbb T^K$ formed by elements $(\lambda_1,\ldots ,\lambda_K)$ satisfying $\lambda_1\cdots\lambda_K=1$. There is a natural continuous action of $\mathbb T^{K-1}$ on $(S^{N-1}_{\mathbb C})^K$, by componentwise multiplication. With this convention, we have:

\begin{lemma}\label{lemm:projpure}
 The map 
$$
 (S^{N-1}_{\mathbb C})^K  \longrightarrow \mathcal P_u((\mathbb C^N)^{\otimes K}), \quad
(z_1, \ldots, z_K) \longmapsto z_1 \otimes \cdots \otimes z_K$$
induces an homeomorphism between $(S^{N-1}_{\mathbb C})^K/\mathbb T^{K-1}$ and $\mathcal P_u((\mathbb C^N)^{\otimes K})$
\end{lemma}

\begin{proof}
 The map in the statement is clearly continuous. Consider now an arbitrary element $v_1\otimes\ldots\otimes v_K\in\mathcal P_u((\mathbb C^N)^{\otimes K})$. Since this element has norm 1, we have:
$$v_1\otimes\ldots\otimes v_K=\frac{1}{||v_1||}v_1\otimes\ldots\otimes\frac{1}{||v_K||}v_K$$
Thus this element belongs to the image of our map, and our map is surjective.

It is clear that the image of two elements lying in the same $\mathbb T^{K-1}$-orbit is the same. Conversely, assume $z_1\otimes\ldots\otimes z_K =z_1'\otimes\ldots\otimes z_K'$. If $z_K,z'_K$ are not colinear, by using appropriate linear forms, we obtain $z_1\otimes\ldots\otimes z_{K-1}=0$, contradicting the norm 1 property. Thus there exists $\alpha_K\in\mathbb T$ such that $z'_K=\alpha_Kz_K$, and by using again an appropriate linear form we see that $z_1\otimes\ldots\otimes z_{K-1}=z'_1 \otimes\ldots\otimes\alpha_Kz'_{K-1}$. Continuing this process, we see that $(z_1,\ldots,z_K)$ and $(z'_1,\ldots,z'_K)$ belong to the same $\mathbb T^{K-1}$-orbit, as needed. Our map is  a continuous bijection between compact spaces, and hence an homeomorphism.   
\end{proof}

We have the following $C^*$-algebraic translation of the previous result, using the notation introduced at the end of the previous section:

\begin{proposition}\label{prop:inv}
 We have a $C^*$-algebra isomorphism 
\begin{align*}
\Psi :  C(\mathcal P_u((\mathbb C^N)^{\otimes K})) & \longrightarrow C((S^{N-1}_{\mathbb C})^K)^{\mathbb T^{K-1}} \\
r_{i_1, i_2, \ldots ,i_K} & \longmapsto a_{i_1,0} \cdots a_{i_K,K-1}
\end{align*}
\end{proposition}

\begin{proof}
Since we have $a_{ic}(z_0,\ldots,z_{K-1})=(z_c)_i$, this is precisely the $C^*$-algebra morphism induced by the homeomorphism found in the previous lemma.
\end{proof}

\section{Matrix models for higher liberated complex spheres}

We now will show that the map in Proposition \ref{prop:mapcrossed} is injective, providing a faithful matrix model for $S^{N-1}_{\mathbb C, K}$. 

Our first result is the connection of the considerations of the previous section with $S^{N-1}_{\mathbb C, K}$, as follows:

\begin{proposition}\label{prop:connec}
There exists a $C^*$-algebra map
$$\Phi: C(\mathcal P_u((\mathbb C^N)^{\otimes K}))\to C(S^{N-1}_{\mathbb C,K}), \quad r_{i_1,\ldots , i_K}\mapsto z_{i_1}\cdots z_{i_K}$$
whose image is the $C^*$-subalgebra of $C(S^{N-1}_{\mathbb C,K})$ generated by the elements of: 
$$\Delta_K=\left\{z_{i_1}^{e_1}\cdots z_{i_s}^{e_s}\Big|\ s \geq 0, \ e_i \in\{1,*\}, \ \#\{e_i=1\}=\#\{e_i=*\}[K]\right\}$$  
\end{proposition}

\begin{proof}
The existence of a morphism $\Phi$ as in the statement follows from Proposition \ref{prop:prespure}, from the defining relations of $C(S^{N-1}_{\mathbb C,K})$, and from Proposition \ref{prop:Khalf}. 

We have to prove that any element of $\Delta_K$ belongs to the image of $\Phi$. So, let $z=z_{i_1}^{e_1}\cdots z_{i_s}^{e_s}\in\Delta_K$. We proceed by induction on $s$. We have 6 cases, as follows:

Case 1: $z=z_{i_1} \cdots z_{i_K}x$, with $x$ monomial in $\{z_i,z_i^*\}$. Then $x\in\Delta_K$ and by the induction we have $x\in Im(\Phi)$, and since $z_{i_1}\cdots z_{i_K}\in Im(\Phi)$, we get $z\in Im(\Phi)$. 

Case 2: $z=z_{i_1}^*\cdots z_{i_K}^*x$, with $x$ monomial in $\{z_i,z_i^*\}$. This is similar to Case 1.

Case 3: $z=z_{i_1}\cdots z_{i_t}z_{j_1}^*\cdots z_{j_t}^*$. If $t>K$ we have $z \in Im(\Phi)$ by Case 1. Otherwise:
$$z=\sum_{i_{t+1}\ldots i_K}z_{i_1}\cdots z_{i_t}z_{i_{t+1}}\cdots z_{i_K}z_{i_K}^*\cdots z_{i_{t+1}}^*z_{j_1}^*\cdots z_{j_t}^*\in Im(\Phi)$$

Case 4: $z=z_{i_1}^*\cdots z_{i_t}^*z_{j_1}\cdots z_{j_t}$. This is similar to Case 3.

Case 5: $z=z_{i_1}\cdots z_{i_t}z_{i_{t+1}}^*x$ with $1\leq t<K$ and $x$ monomial in $\{z_i,z_i^*\}$. We have:
$$z=z_{i_1}\cdots z_{i_t}z_{i_{t+1}}^*x=\sum_{\alpha_3,\ldots,\alpha_{t+1}}
z_{i_1}\cdots z_{i_t}z_{i_{t+1}}^*z_{\alpha_{t+1}}^*\cdots z_{\alpha_3}^* z_{\alpha_3}\cdots z_{\alpha_{t+1}}x$$
Hence the elements $y=z_{\alpha_3}\cdots z_{\alpha_{t+1}}x$ belong to $\Delta_K$, and by induction,  belong to $Im(\Phi)$. Thus by using Case 3 we conclude that $z\in Im(\Phi)$. 

Case 6: $z=z_{i_1}^*\cdots z_{i_t}^*z_{i_{t+1}}x$ with $1 \leq t<K$ and $x$ monomial in $\{z_i,z_i^*\}$. We can proceed here as in Case 5, and by using Case 4, we obtain $z\in Im(\Phi)$.
\end{proof}

The above result shows in particular that for $x,y\in\Delta_K$, we have $[x,y]=0$, since these elements belong to the image of a commutative algebra, and this provides an alternative description of $C(S_{\mathbb C, K}^{N-1})$.

\begin{corollary}\label{cor:Delta_K}
We have $$C(S_{\mathbb C, K}^{N-1})= C(S_{\mathbb C, +}^{N-1})/ \Big/\Big<[x,y]=0, \ x,y\in\Delta_K\Big>$$
\end{corollary}

 We will need:

\begin{proposition}\label{prop:ZK-grading}
The algebra $C(S^{N-1}_{\mathbb C,K})$ has a natural $\mathbb Z_K$-grading whose $0$-component is the $C^*$-subalgebra generated by the elements of $\Delta_K$.
\end{proposition}

\begin{proof}
The group $\mu_K$ of $K$-th roots of unity acts on $C(S^{N-1}_{\mathbb C, K})$ by $\omega\cdot z_i=\omega z_i$. Let us set:
$$C(S^{N-1}_{\mathbb C,K})_j=\left\{a\in C(S^{N-1}_{\mathbb C,K})\Big|\ \omega\cdot a=\omega^ja, \forall\omega\right\}$$
We obtain in this way an algebra $\mathbb Z_K$-grading, as follows:
$$C(S^{N-1}_{\mathbb C,K})=\bigoplus_{j=0}^{K-1}C(S^{N-1}_{\mathbb C,K})_j$$
Since $\mathcal O(S^{N-1}_{\mathbb C,K})_0$ is the $*$-subalgebra generated by the elements of $\Delta_K$ and is dense in $C(S^{N-1}_{\mathbb C,K})_0$, we are done.
\end{proof}

We now have all the ingredients to prove the following result:

\begin{theorem}\label{thm:faithmodel}
 There exists a faithful matrix model
$$C(S^{N-1}_{\mathbb C,K}) \rightarrow M_K(C((S^{N-1}_{\mathbb C})^K))$$
\end{theorem}

\begin{proof}
 We first show that the map $\pi$ from Proposition \ref{prop:mapcrossed} is injective.
By  Proposition \ref{prop:ZK-grading}   $C(S^{N-1}_{\mathbb C,K})$ is $\mathbb Z_K$-graded. The $C^*$-algebra $C((S^{N-1}_{\mathbb C})^K)\rtimes\mathbb Z_K$ is $\mathbb Z_K$-graded as well, with:   
$$(C((S^{N-1}_{\mathbb C})^K)\rtimes\mathbb Z_K)_j=C((S^{N-1}_{\mathbb C})^K)\otimes\tau^j$$
Since $\pi$ preserves the grading, by a standard argument it is enough to show that the restriction of $\pi$ to the zero component is injective. We use the maps $\Psi,\Phi$ from Proposition \ref{prop:inv} and Proposition \ref{prop:connec}. By Propositions \ref{prop:connec} and \ref{prop:ZK-grading}, the image of $\Phi$ is the $0$-component of $C(S^{N-1}_{\mathbb C, K})$. We have: 
$$\pi\Phi(r_{i_1\ldots i_K})=a_{i_1,0}\cdots a_{i_K,K-1}\otimes 1=\Psi(r_{i_1, \ldots,i_K})\otimes 1$$

Thus $\pi\Phi=\Psi\otimes 1$, and since $\Psi$ is injective, we conclude that $\pi$ is injective on the algebra $C(S^{N-1}_{\mathbb C,K})_0=Im(\Phi)$, hence on $C(S^{N-1}_{\mathbb C,K})$. 

The theorem now follows by using the standard embedding $$C((S^{N-1}_{\mathbb C})^K) \rtimes \mathbb Z_K \subset M_K(C((S^{N-1}_{\mathbb C})^K))$$
 obtained using the permutation matrix of a $K$-cycle.
\end{proof}

\begin{corollary}\label{cor:irrSNK}
 The representations of $C(S^{N-1}_{\mathbb C,K})$ have the following properties:
\begin{enumerate}
\item There exist irreducible representations of dimension $K$.

\item Any irreducible representation is finite dimensional, of dimension $\leq K$.
\end{enumerate}
In particular, for $K\neq K'$, the $C^*$-algebras $C(S^{N-1}_{\mathbb C,K})$ and $C(S^{N-1}_{\mathbb C,K'})$ are not isomorphic.
\end{corollary}

\begin{proof}
We use the standard embedding mentioned above, namely:
$$C ((S^{N-1}_{\mathbb C})^K)\rtimes\mathbb Z_K\subset M_K(C((S^{N-1}_{\mathbb C})^K), \quad a\otimes\tau^j\mapsto \sum_c\tau^c(a)E_{c,c+j}$$

(1) Any element $x =(x_0,\ldots,x_{K-1})\in(S^{N-1}_{\mathbb C})^K$ defines, by evaluation composed with $\pi$, a representation $\rho_x:C (S^{N-1}_{\mathbb C,K})\to M_K(\mathbb C)$, given by:
$$z_i\to\sum_c\tau^c(a_{i,0})(x)E_{k,k+1}=\sum_ca_{ic}(x)E_{c,c+1}=
\sum_c(x_c)_iE_{c,c+1}$$

Now choose $x$ such that $(x_c)_1=\frac{1}{\sqrt{N}}$, $(x_c)_2= \frac{\xi_c}{\sqrt{N}}$ for any $c$, with the elements $\xi_c \in \mathbb T$ being pairwise distinct. Then the commutant of the matrices $\rho_x(z_1),\rho_x(z_2)$ is reduced to the set of scalar matrices, and so our representation is irreducible.

(2) This follows from the theorem and Proposition \ref{prop:faithimplX=XK}, and the last assertion follows from (1) and (2).
\end{proof}

As a useful consequence of the proof of Theorem \ref{thm:faithmodel}, we also record, for future use:

\begin{corollary}\label{cor:Phibijective}
 The $C^*$-algebra map
$$\Phi: C(\mathcal P_u((\mathbb C^N)^{\otimes K}))\to C(S^{N-1}_{\mathbb C,K})_0, \quad r_{i_1,\ldots , i_K}\mapsto z_{i_1}\cdots z_{i_K}$$
is an isomorphism.
\end{corollary}

\section{The $1/K$-classical version of a noncommutative manifold}

We now generalize the previous construction to more general objects $X \subset S^{N-1}_{\mathbb C, +}$.
For this, we first need to introduce some vocabulary.
Recall that the $\mathbb Z_K$-grading on $C(S^{N-1}_{\mathbb C,K})$ from Proposition \ref{prop:ZK-grading} comes from the $\mu_K$-action on 
$C(S^{N-1}_{\mathbb C,K})$ defined by $\omega\cdot z_i=\omega z_i$. This action is in fact defined on $C(S^{N-1}_{\mathbb C,+})$. 

\begin{definition}
 We say that $X\subset S^{N-1}_{\mathbb C,+}$ is $K$-symmetric if the above $\mu_K$-action on $C(S^{N-1}_{\mathbb C,+})$ induces a $\mu_K$-action on $C(X)$.
\end{definition}

For example $S^{N-1}_{\mathbb C,K}$ is itself $K$-symmetric.

\begin{definition}
For $X\subset S^{N-1}_{\mathbb C,+}$, assumed to be $K$-symmetric, the $1/K$-classical version of $X$ is defined by 
$$X_{1/K-class}=X\cap S^{N-1}_{\mathbb C,K}$$
or, in other words, by
$$C(X_{1/K-class}) = C(X)\Big/\Big< [z_{i_1}\cdots z_{i_K}, z_{j_{1}} \cdots z_{j_K}]=0
=  [z_{i_1} \cdots z_{i_K}, z_{j_1}^* \cdots z_{j_K}^*]\Big>$$
\end{definition}

Clearly the $1/K$-classical version of $S^{N-1}_{\mathbb C,+}$ is $S^{N-1}_{\mathbb C,K}$, and if $X$ is algebraic, so is $X_{1/K-class}$.

\begin{remark}\label{rem: patho}
The symmetry assumption is here to avoid some pathologies. Indeed, for $X=S^{N-1}_{\mathbb R,+}$, which is not $K$-symmetric for $K\geq 3$, our definition would give, for $K \geq 3$, that the $1/K$-classical version is $S^{N-1}_{\mathbb R}$, the classical version. This fits with the fact that in the real case, for $K \geq 3$, the $K$-half liberation procedure does not produce any new sphere in the real case (Remark \ref{rem:nonreal}), and only the $K=2$ case is allowed.
\end{remark}



 It is clear that if $X \subset S^{N-1}_{\mathbb C,+}$ is $K$-symmetric, then $X_{1/K-class}$ is also $K$-symmetric.

We will also say that a subset $T \subset (S^{N-1}_{\mathbb C})^K$ is symmetric if it is stable under the cyclic action of $\mathbb Z_K$.


Our aim now is to construct a faithful matrix model for $X  \subset S^{N-1}_{\mathbb C,K}$ $K$-symmetric. We will use the following tool.

\begin{definition}
 We denote by  $\gamma$ the linear endomorphism of $C(S^{N-1}_{\mathbb C,K})$ defined by $$\gamma(a)=\sum_{i=1}^nz_iaz_i^*$$
\end{definition}

The main properties of $\gamma$ are summarized in the following lemma, where we use the map $\Phi$ from Proposition \ref{prop:ZK-grading}.
 
\begin{lemma}\label{gamma}
 The endomorphism $\gamma$ preserves the $\mathbb Z_K$-grading of $C(S^{N-1}_{\mathbb C,K})$, and induces a $*$-algebra automorphism of $C(S^{N-1}_{\mathbb R,K})_0$. Moreover the following diagram commutes 
\begin{equation*}
\begin{CD}
  C(\mathcal P_u((\mathbb C^N)^{\otimes K})) @>\Phi>>   C(S^{N-1}_{\mathbb C,K})_0 \\
@VV\tau V @VV\gamma V \\
 C(\mathcal P_u((\mathbb C^N)^{\otimes K})) @>\Phi>>   C(S^{N-1}_{\mathbb C,K})_0 
\end{CD}
\end{equation*}
where $\tau$ is the cyclic automorphism induced given by $\tau(r_{i_1\ldots i_K})=r_{i_Ki_1 \ldots i_{K-1}}$. Hence there is a bijective correspondence between $\gamma$-stable ideals of $C(S^{N-1}_{\mathbb C,K})_0$ and symmetric closed subsets of $\mathcal P_u((\mathbb C^N)^{\otimes K})$.
\end{lemma}
 
\begin{proof}
 Since $\gamma$ commutes with the $\mu_K$-action, it indeed preserves $\mathbb Z_K$-grading of $C(S^{N-1}_{\mathbb C,K})$. We have also $\gamma(1)=1$, and $\gamma$ commutes with the involution. We have, using Proposition \ref{prop:Khalf},
$$\gamma(z_{i_1} \cdots z_{i_K}) = \sum_i z_i z_{i_1} \cdots z_{i_K} z_i^*= \sum_i z_{i_K}z_i^*z_iz_{i_1} \cdots z_{i_{K-1}} = z_{i_K}z_{i_1} \cdots z_{i_{K-1}}$$
$$\gamma(z_{i_K}^* \cdots z_{i_1}^*)= \sum_i z_i z_{i_K}^* \cdots z_{i_1}^* z_i^*=\sum_i z_{i_{K-1}}^* \cdots z_{i_1}^*z_i^*z_i z_{i_K}^*=z_{i_{K-1}}^* \cdots z_{i_1}^* z_{i_K}^*$$
From this, and using Proposition \ref{prop:Khalf}, one shows by induction that $\gamma$ is an algebra morphism on $\mathcal O(S^{N-1}_{\mathbb C,K})_0$, and hence on $C(S^{N-1}_{\mathbb C,K})_0$. We see also that
 $\gamma\Phi$ and $\Phi \tau$ coincide on the $*$-subalgebra generated by the elements $r_{i_1, \ldots, i_K}$, and we conclude by density that the diagram commutes. This shows simultaneously that $\gamma$ induces a $*$-algebra automorphism of $C(S^{N-1}_{\mathbb R,K})_0$ (since $\Phi$ is an isomorphism), and the last assertion follows as well.
\end{proof}

Our next technical result expresses the property of being $K$-symmetric in term of ideals.

\begin{proposition}\label{prop:ideals}
 Let $X \subset S^{N-1}_{\mathbb C, K}$ with $C(X)=C(S^{N-1}_{\mathbb C, K})/I$. 
The following assertions are equivalent.
\begin{enumerate}
\item $X$ is $K$-symmetric.
\item The ideal $I$ is $\mathbb Z_K$-graded, i.e.
$$I = I_0 + I_1 + \cdots + I_{K-1}, \ {\rm with} \ I_l= I \cap C(S^{N-1}_{\mathbb C, K})_l$$  
\item There exists some  $\gamma$-stable ideal $J \subset C(S^{N-1}_{\mathbb C, K})_0$ such that
 $$I = \langle J \rangle =J+  C(S^{N-1}_{\mathbb C, K})_1J + \cdots + C(S^{N-1}_{\mathbb C, K})_{K-1}J$$
\end{enumerate}
\end{proposition}

\begin{proof}
 The equivalence of (1) and (2) is well-known, since the $\mathbb Z_K$-grading arises from the $\mu_K$-action, while (3) $\Rightarrow$ (2) is obvious. Assume that (2) holds, and put $J:= I_0=I \cap C(S^{N-1}_{\mathbb C, K})_0$. Then $J$ is an ideal in $C(S^{N-1}_{\mathbb C, K})_0$, and is $\gamma$-stable since $I$ is an ideal. It is clear that we have $J +  C(S^{N-1}_{\mathbb C, K})_1J + \cdots + C(S^{N-1}_{\mathbb C, K})_{K-1}J \subset I$. To prove the reverse inclusion, consider $a \in I_l=I \cap C(S^{N-1}_{\mathbb C, K})_l$. We have
$$a= \sum_{i_1, \ldots ,i_l}z_{i_1} \cdots z_{i_{l}} z_{i_{l}}^* \cdots z_{i_{1}}^*a \in  C(S^{N-1}_{\mathbb C, K})_{l}I_0= C(S^{N-1}_{\mathbb C, K})_{l}J$$
and we are done.
\end{proof}

We arrive at the main result of the section, which generalizes the injectivity of the map in Proposition \ref{prop:mapcrossed}.

\begin{theorem}\label{thm:machine}
 Let $X \subset S^{N-1}_{\mathbb C, K}$ be $K$-symmetric. Then there exists a symmetric compact subspace  $T \subset (S^{N-1}_{\mathbb C})^K$ such that the morphism $\pi$ of Proposition \ref{prop:mapcrossed} induces an injective morphism $C(X) \rightarrow C(T) \rtimes \mathbb Z_K$.

The space $T$ is constructed as follows.
\begin{enumerate}
 \item 
 Write  $C(X)=C(S^{N-1}_{\mathbb C, K})/\langle J \rangle$ as in Proposition \ref{prop:ideals}, with $J \subset C(S^{N-1}_{\mathbb C, K})_0$ a $\gamma$-stable ideal. 
\item Consider the isomorphism $\Phi : C(\mathcal P_u((\mathbb C^N)^{\otimes K}))\to C(S^{N-1}_{\mathbb C,K})_0$ from Proposition \ref{prop:connec}: we get an ideal $\Phi^{-1}(J)$ in $C(\mathcal P_u((\mathbb C^N)^{\otimes K}))$, which is the ideal of vanishing functions on a symmetric compact subset $T_0 \subset \mathcal P_u((\mathbb C^N)^{\otimes K})$.
\item The symmetric compact subspace $T \subset (S^{N-1}_{\mathbb C})^K$ is then defined by $T=p^{-1}(T_0)$, where 
$p : (S^{N-1}_{\mathbb C})^K \rightarrow \mathcal P_u((\mathbb C^N)^{\otimes K})$ is the canonical surjection (see Lemma \ref{lemm:projpure}).
\end{enumerate}
\end{theorem}

\begin{proof}
 The space $T$ is constructed following the procedure in the statement of the proposition. Since $T$ is symmetric, we can form the crossed product $C(T)\rtimes \mathbb Z_K$, and using restriction of functions, we get the $*$-algebra map 
$$C(S^{N-1}_{\mathbb C, K}) \rightarrow C(T)\rtimes \mathbb Z_K, \ z_i \mapsto a_{i0} \otimes \tau$$
that we still call $\pi$. Since $\pi$ is still morphism of $\mathbb Z_K$-graded algebras (as in the proof of Theorem \ref{thm:faithmodel}), then $Ker(\pi)$ is $\mathbb Z_K$-graded and, by Proposition \ref{prop:ideals}, it is enough to show that the $Ker(\pi) \cap C(S^{N-1}_{\mathbb C, K})_0$ equals the ideal $J$ that we started with. So let $a \in C(S^{N-1}_{\mathbb C, K})_0$, with $a=\Phi(f)$ and $f \in C(  P_u((\mathbb C^N)^{\otimes K}))$. Then, again similarly to the proof of Theorem \ref{thm:faithmodel}, we have $\pi(a)= \pi(\Phi(f))= \Psi(f)_{|T} \otimes 1= fp_{|T} \otimes 1$, hence $a \in Ker(\pi)$ if and only $fp$ is zero on $T$, if and only if $f$ vanishes on $T_0$, if and only if $f \in \Phi^{-1}(J)$, hence $a \in Ker(\pi)$ if and only if $a \in J$. This concludes the proof.
\end{proof}

Starting now from  $X \subset S^{N-1}_{\mathbb C,+}$ assumed to be $K$-symmetric, Theorem \ref{thm:machine} applied to $X_{1/K-class}$, together with the standard matrix model of the crossed product, yields:

\begin{theorem}\label{thm:faith1/K}
 Let $X \subset S^{N-1}_{\mathbb C,+}$ be $K$-symmetric. Then there exists a faithful matrix model
$$C(X_{1/K-class}) \longrightarrow M_K(C(T))$$
where $T$ is an appropriate symmetric compact subset of $(S^{N-1}_{\mathbb C})^K$. In particular we have
$X_{1/K-class} \subset X^{(K)}$. 
\end{theorem}

We end the the section by discussing  a possible future research direction.
The above considerations strongly suggest the following definition:

\begin{definition}
Associated to $X\subset S^{N-1}_{\mathbb C,+}$ is the space $X^{\{K\}}\subset S^{N-1}_{\mathbb C,+}$ given by
$$C(X^{\{K\}})\subset C(X)^{\otimes K}\rtimes\mathbb Z_K$$
where the group $\mathbb Z_K$ acts cyclically on the tensor product $C(X)^{\otimes K}$, and $C(X^{\{K\}})$ is the $C^*$-subalgebra generated by the elements $z_i\otimes 1 \otimes \ldots \otimes 1 \otimes \tau$.
\end{definition}

Indeed, we have shown that $(S^{N-1}_{\mathbb C})^{\{K\}}= S^{N-1}_{\mathbb C,K}$. Starting with non-classical $X$, finding a presentation of $C(X^{\{K\}})$ from one of $C(X)$ seems to be more difficult: the general scheme of the proof of Theorem \ref{thm:faithmodel} is still valid, but the geometric techniques from Section 4 needed to study the grade $0$ part are no longer available. We believe that this an interesting problem.

\section{Quantum groups}

We now apply the previous considerations to construct new classes of compact quantum groups.

\begin{definition}
 The quantum group $ U_{N,K}^*\subset U^{+}_{N}$ defined by
$$C(U_{N,K}^*)= C(U_N^+)\Big/\Big< [u_{i_1j_1}\cdots u_{i_Kj_K}, u_{k_{1}l_1} \cdots u_{k_Kl_K}]=0
\Big>$$
is called the $K$-half-liberated unitary quantum group.
\end{definition}

It is straightforward that this indeed defines a quantum group. We did not include the second family of relations from the definition of $S^{N-1}_{\mathbb C, K}$, since they follow easily from the other relations:

\begin{proposition}\label{prop: addrelUNK}
In $C(U_{N,K}^*)$, we have as well
 $$[u_{i_1j_1}\cdots u_{i_Kj_K}, u_{k_{1}l_1}^* \cdots u_{k_Kl_K}^*]=0$$
\end{proposition}

\begin{proof}
This follows from the well-known fact that if the coefficients of two unitary representations $(u_{ij})$, $(v_{kl})$ of a quantum group pairwise commute, then the coefficients of $(u_{ij})$ and  $(v_{kl}^*)$ also pairwise commute.
To check this, start with the relations $u_{ij}v_{kl}=v_{kl}u_{ij}$, multiply on the right by $v_{pl}^*$ and sum over $l$ to get $\delta_{kp}u_{ij}= \sum_l v_{kl}u_{ij}v_{pl}^*$. Now multiplying on the left by $v_{kq}^*$ and summing over $k$ yields $v_{pq}^*u_{ij}=u_{ij}v_{pq}^*$, as needed.
\end{proof}

We then have 
$$U_{N,1}^*=U_N, \quad U_{N,2}^*= U_N^{**}$$
where $U_N^{**}$ is the half-liberated unitary quantum group from \cite{bdu}. Similarly to Remark \ref{rem:nonreal}, the orthogonal version of the above construction would not lead to any new quantum group.

More generally, recall from Section 1 that we have an embedding $U_N^+ \subset S^{N^2-1}_{\mathbb C,+}$ coming from the presentation
$$C(U_N^+)=C(S^{N^2-1}_{\mathbb C,+})\Big/\Big(zz^*=z^*z=\bar{z}z^t=z^t\bar{z}=\frac{1}{N}\cdot1_N\Big)$$
It is clear from this presentation  that $U_N^+\subset S^{N^2-1}_{\mathbb C,+}$ is $K$-symmetric, and that we have
$$U_{N,K}^* \subset S^{N^2-1}_{\mathbb C,K}, \ U_{N,K}^*= (U_N^+)_{1/K-class}$$
We therefore can use the machinery of Theorem \ref{thm:machine}, and after some tedious identifications, we get:

\begin{theorem}\label{thm:UNK*}
 We have an injective morphism of $C^*$-algebras
$$C(U_{N,K}^*) \longrightarrow C(U_{N}^{K}) \rtimes \mathbb Z_K $$
where $\mathbb Z_K$ acts cyclically on the product $U_N^K$.
\end{theorem}

The above embedding is compatible with the respective comultiplications as well, so it is possible, similarly to \cite{bdu}, to describe the irreducible representations of the quantum group $U_{N,K}^*$ in terms of those of the compact group $U_N^K$. Note that $U_{N,K}^*$ is an easy quantum group as well, see the next section for more details.

To conclude this section, let us point out that the considerations of Section 6 may be applied to any $K$-symmetric quantum subgroup $G \subset U_{N}^+$, yielding a $1/K$-classical version of $G$. This applies
as well to diagonal dual subgroups of $U_N^+$, but in that precise framework, much more direct arguments can be used to prove the analogue of the previous theorem.

\section{The limit cases}

We now introduce  the ``limit'' cases of our $K$-half-liberated spheres and quantum groups.  
The following definition is inspired by the second relations in Proposition \ref{prop:Khalf}, which do not depend on $K$.

\begin{definition}
The strong half-liberated complex sphere is defined by
$$C(S^{N-1}_{\mathbb C,\infty})=C(S^{N-1}_{\mathbb C,+})\Big/\Big<\{z_i^*z_j,z_jz_i^*\}\ all\ commute\Big>$$
\end{definition}

\begin{proposition}\label{prop:incluinfty}
 We have, for any $K$,  strict inclusions $S^{N-1}_{\mathbb C,K}\subset S^{N-1}_{\mathbb C,\infty}\subset S^{N-1}_{\mathbb C,*}$.
\end{proposition}

\begin{proof}
The first inclusion follows from Proposition \ref{prop:Khalf}. This  is a strict inclusion since by Corollary \ref{cor:irrSNK}, any irreducible representation of $C(S^{N-1}_{\mathbb C,K})$ has dimension $\leq K$, while since  $S^{N-1}_{\mathbb C,K}\subset S^{N-1}_{\mathbb C,\infty}$ for any $K$,  Corollary \ref{cor:irrSNK} implies that  $C(S^{N-1}_{\mathbb C,\infty})$ has irreducible representations of any possible finite dimension. The second inclusion  comes from Proposition \ref{prop:incluhalf}. To prove strictness of this second inclusion, we use ideas from the theory of graded twisting \cite{bny}. Consider the free product $C^*$-algebra $C(S^{N-1}_{\mathbb R, *})*C(S^{N-1}_{\mathbb R, *})$, with the canonical generators of the first copy denoted $x_1, \ldots , x_N$, and the generators of the second copy denoted $y_1, \ldots, y_N$. Denote by $\theta$ the involutive automorphism of $C(S^{N-1}_{\mathbb R, *})*C(S^{N-1}_{\mathbb R, *})$ that exchanges $x_i$ and $y_i$, and form the corresponding crossed product $C(S^{N-1}_{\mathbb R, *})*C(S^{N-1}_{\mathbb R, *})\rtimes \mathbb Z_2$. Then, similarly to Example 3.7 in \cite{bny}, there is a morphism
$$\rho : C(S^{N-1}_{\mathbb C,*}) \rightarrow C(S^{N-1}_{\mathbb R, *})*C(S^{N-1}_{\mathbb R, *})\rtimes \mathbb Z_2, \ z_i, z_i^* \mapsto x_i \otimes \theta, y_i \otimes \theta$$
We then have
$$\rho(z_iz_j^*z_k^*z_l)= x_ix_jy_ky_l \otimes 1, \ \rho(z_k^*z_lz_iz_j^*)=y_ky_lx_ix_j \otimes 1$$
Hence if we had $z_iz_j^*z_k^*z_l=z_k^*z_lz_iz_j^*$ in $C(S^{N-1}_{\mathbb C,*})$, the relations $x_ix_jy_ky_l= y_ky_lx_ix_j$ would hold in $C(S^{N-1}_{\mathbb R, *})*C(S^{N-1}_{\mathbb R, *})$, which is not true, by general properties of the free product \cite{nsp}. It follows that the canonical morphism $C(S^{N-1}_{\mathbb C,*}) \rightarrow C(S^{N-1}_{\mathbb C,\infty})$ is not injective, and our inclusion is strict.
\end{proof}


We remark that
Corollary \ref{cor:Delta_K} suggests the definition of another limit sphere
$$C(S_{\mathbb C, \infty^{-}}^{N-1})= C(S_{\mathbb C, +}^{N-1})/ \Big/\Big<[x,y]=0, \ x,y\in\Delta_\infty\Big>$$
where 
$$\Delta_\infty=\left\{z_{i_1}^{e_1}\ldots z_{i_s}^{e_s}\Big|\ s \geq 0, \ e_i \in\{1,*\}, \ \#\{e_i=1\}=\#\{e_i=*\}\right\}$$ 
We have $S^{N-1}_{\mathbb C,K}\subset S^{N-1}_{\mathbb C,\infty^{-}}\subset S^{N-1}_{\mathbb C,\infty}\subset S^{N-1}_{\mathbb C,*}$. It is unclear to us whether the inclusion $S^{N-1}_{\mathbb C,\infty^{-}} \subset S^{N-1}_{\mathbb C,\infty}$ is strict, and if the inclusion is strict, it is as well unclear whether $C(S_{\mathbb C, \infty^{-}}^{N-1})$ has a finite presentation. Another interesting open question is: do we have $S^{N-1}_{\mathbb C , \infty^{-}} = \cup_{K\geq 1} S^{N-1}_{\mathbb C, K}$?

At the quantum group level, the corresponding definition of the limit quantum group is as follows:

\begin{definition}
The strong half-liberated unitary quantum group is defined by
$$C(U_{N, \infty})=C(U_N^{+})\Big/\Big<\{u_{ij}^*u_{kl},u_{ij}u_{kl}^*\}\ all\ commute\Big>$$
\end{definition}

Similarly to Proposition \ref{prop:incluinfty}, we have:

\begin{proposition}
 We have, for any $K$, strict  inclusions $U_{N,K}\subset U_{N,\infty}\subset U_{N,*}$.
\end{proposition}

Let us now explain briefly that $U_{N, \infty}$ is an easy quantum group. For $k,l\geq 0$, 
let $P(k,l)$ be the set of partitions between an upper row of $k$ points, and a lower row of $l$ points, with each leg colored black or white, and with $k,l$ standing for the corresponding ``colored integers''. We have then the following notion:

\begin{definition}
A category of partitions is a collection of sets $D=\bigcup_{kl}D(k,l)$, with $D(k,l)\subset P(k,l)$, which contains the identity, and is stable under:
\begin{enumerate}
\item The horizontal concatenation operation $\otimes$.

\item The vertical concatenation $\circ$, after deleting closed strings in the middle.

\item The upside-down turning operation $*$ (with reversing of the colors).
\end{enumerate}
\end{definition}

Here the vertical concatenation operation assumes of course that the colors match. Regarding the identity, the precise condition is that $D(\circ,\circ)$ contains the ``white'' identity $|^{\hskip-1.3mm\circ}_{\hskip-1.3mm\circ}$\,. By using (3) we see that $D(\bullet,\bullet)$ contains the ``black'' identity $|^{\hskip-1.3mm\bullet}_{\hskip-1.3mm\bullet}$\,, and then by using (1) we see that each $D(k,k)$ contains its corresponding (colored) identity.

As explained in  \cite{tw2}, such categories produce quantum groups. To be more precise, associated to any partition $\pi\in P(k,l)$ is the following linear map:
$$T_\pi(e_{i_1}\otimes\ldots\otimes e_{i_k})=\sum_{j:\ker(^i_j)\leq\pi}e_{j_1}\otimes\ldots\otimes e_{j_l}$$

Here the kernel of a multi-index $(^i_j)=(^{i_1\ldots i_k}_{j_1\ldots j_l})$ is the partition obtained by joining the sets of equal indices. Thus, the condition $\ker(^i_j)\leq\pi$ simply tells us that the strings of $\pi$ must join equal indices. With this construction in hand, we have:

\begin{definition}
A compact quantum group $G\subset U_N^+$ is called easy when
$$Hom(u^{\otimes k},u^{\otimes l})=span\left(T_\pi\Big|\pi\in D(k,l)\right)$$
for any $k,l$, for a certain category of partitions $D\subset P$.
\end{definition}

In other words, the easiness condition states that the Schur-Weyl dual of $G$ comes in the ``simplest'' possible way: from partitions. As a basic example, according to an old result of Brauer \cite{bra, gv}, the group $G=U_N$ is easy, with $D=P_2$ being the category of ``color-matching'' pairings . Easy as well is $U_N^+$, with $D=NC_2\subset P_2$ being the category of noncrossing color-matching pairings. See \cite{bsp}, \cite{fre}, \cite{rwe}, \cite{tw2}.

With these notions in hand, here is now our main statement here:

\begin{theorem}
The unitary quantum group $U_{N,\infty}$ is easy, coming from the following bicolored partitions:
$$
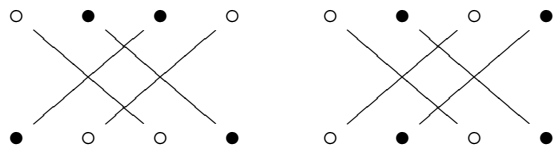
\xymatrix@R=12mm@C=5mm{\circ\ar@{-}[drr]&\bullet\ar@{-}[drr]&\bullet\ar@{-}[dll]&\circ\ar@{-}[dll]\\\bullet&\circ&\circ&\bullet} \quad \quad 
\xymatrix@R=12mm@C=5mm{\circ\ar@{-}[drr]&\bullet\ar@{-}[drr]&\circ\ar@{-}[dll]&\bullet\ar@{-}[dll]\\\circ&\bullet&\circ&\bullet}
$$
\end{theorem}

\begin{proof}
 Denote by $U$ the Hilbert space $\mathbb C^N$, endowed with its canonical basis. The linear maps corresponding to the two above diagram respectively are:
$$U \otimes \bar{U} \otimes \bar{U} \otimes U \rightarrow \bar{U} \otimes U \otimes U \otimes \bar{U}, \ e_i \otimes \bar{e_j} \otimes \bar{e_k} \otimes e_l \mapsto \bar{e_k} \otimes e_l \otimes e_i \otimes \bar{e_j}$$
  $$U \otimes \bar{U} \otimes  U \otimes \bar{U} \rightarrow U \otimes \bar{U} \otimes U \otimes \bar{U},\ e_i \otimes \bar{e_j} \otimes e_k \otimes \bar{e_l} \mapsto e_k \otimes \bar{e_l} \otimes e_i \otimes \bar{e_j}$$
It follows from the defining relations in $U_{N, \infty}$ that these are morphisms in the representation category of $U_{N, \infty}$. Conversely,   if $G \subset U_N^+$ is quantum group such that above morphisms are morphisms in the representation category of $G$, we get the following relations in $C(G)$:
$u_{ij}u_{kl}^*u_{pq}^*u_{rs}= u_{pq}^*u_{rs}u_{ij}u_{kl}^*$ and $u_{ij}u_{kl}^*u_{pq}u_{rs}^*= u_{pq}u_{rs}^*u_{ij}u_{kl}^*$. The second relations give in particular $u_{ij}u_{kl}^*u_{kq}u_{rs}^*= u_{kq}u_{rs}^*u_{ij}u_{kl}^*$, and summing over $k$, this gives 
$\delta_{lq}u_{ij}u_{rs}^*= \sum_k u_{kq}u_{rs}^*u_{ij}u_{kl}^*$. Multiplying by $u_{tl}$ on the right and summing over $l$, this gives  $u_{ij}u_{rs}^*u_{tq}= u_{tq}u_{rs}^*u_{ij}$, the defining relation of $U_{N,*}$. We get from Proposition \ref{prop:incluhalf} that the defining relations of $U_{N, \infty}$ are satisfied, so that $G\subset U_{N, \infty}$.

The above discussion and Tannakian duality show that the representation category of $U_{N, \infty}$ is generated by the partitions in the statement, and hence is an easy quantum group. 
\end{proof}

Note as well that, as already said in the previous section, each quantum group $U_{N,K}$ is easy, coming from the following crossing diagram in $\mathcal P(2K,2K)$:

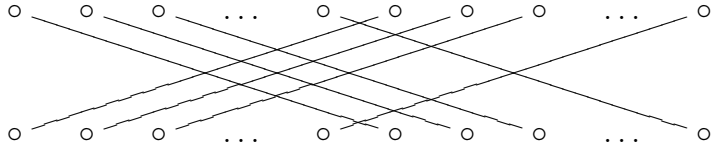
$$\xymatrix@R=12mm@C=5mm{\circ \ar@{-}[drrrrr]& \circ \ar@{-}[drrrrr] & \circ \ar@{-}[drrrrr]& \ldots & \circ \ar@{-}[drrrrr]& \circ\ar@{-}[dlllll]&\circ\ar@{-}[dlllll]&\circ\ar@{-}[dlllll] & \ldots & \circ\ar@{-}[dlllll] \\ \circ & \circ & \circ& \ldots & \circ & \circ&\circ&\circ& \ldots & \circ } 
$$

The embedding of Theorem \ref{thm:UNK*} easily enables one to show that $U_{N,K}$ is coamenable for finite $K$. So we have the following question:

\begin{question}
 Is the compact quantum group $U_{N,\infty}$ coamenable?
\end{question}

We believe that the answer is yes, but we have no proof.  We also think that $U_{N, \infty}$ could be a kind of ``largest coamenable version'' of $U_N$, but here  we have no precise conjectural statement.


\begin{thebibliography}{99}

\bibitem{ba0} T. Banica, Liberations and twists of real and complex spheres, {\em J. Geom. Phys.} {\bf 96} (2015), 1--25.

\bibitem{ba1}T. Banica, Half-liberated manifolds, and their quantum isometries, {\em Glasg. Math. J.}, to appear.

\bibitem{ba2}T. Banica, Liberation theory for noncommutative homogeneous spaces, {\em Ann. Fac. Sci. Toulouse Math.}, to appear.




\bibitem{bago} T. Banica, D. Goswami, Quantum isometries and noncommutative spheres, {\em Comm. Math. Phys.} {\bf 298} (2010), no. 2, 343-356.


\bibitem{bsp}T. Banica and R. Speicher, Liberation of orthogonal Lie groups, {\em Adv. Math.} {\bf 222} (2009), 1461--1501.

\bibitem{bve}T. Banica and R. Vergnioux, Invariants of the half-liberated orthogonal group, {\em Ann. Inst. Fourier} {\bf 60} (2010), 2137--2164.


\bibitem{bdd}J. Bhowmick, F. D'Andrea and L. Dabrowski, Quantum isometries of the finite noncommutative geometry of the standard model, {\em Comm. Math. Phys.} {\bf 307} (2011), 101--131.


\bibitem{bi1}J. Bichon, Half-liberated real spheres and their subspaces, {\em Colloq. Math.} {\bf 144} (2016), 273--287.


\bibitem{bdu}J. Bichon and M. Dubois-Violette, Half-commutative orthogonal Hopf algebras, {\em Pacific J. Math.} {\bf 263} (2013), 13--28. 

\bibitem{bny} J. Bichon, S. Neshveyev and M. Yamashita, 
Graded twisting of categories and quantum groups by group actions, {\em Ann. Inst. Fourier}, to appear. 



\bibitem{bra}R. Brauer, On algebras which are connected with the semisimple continuous groups, {\em Ann. of Math.} {\bf 38} (1937), 857--872.

\bibitem{chi}A. Chirvasitu, Residually finite quantum group algebras, {\em J. Funct. Anal.} {\bf 268} (2015), 3508--3533.



\bibitem{dix}
J. Dixmier,
$C^*$-algebras, North-Holland Mathematical Library 15. North-Holland Publishing Co. (1977).



\bibitem{fre}A. Freslon, On the partition approach to Schur-Weyl duality and free quantum groups, preprint 2014.

 \bibitem{gv} P. Glockner, W. von Waldenfels, The relations of the noncommutative coefficient algebra of the unitary group,  {\em Lecture Notes in Math.} {\bf 1396} (1989), 182--220. 


\bibitem{har} J. Harris, Algebraic geometry. A first course, Graduate Texts in Mathematics 133, Springer-Verlag (1995).


\bibitem{kra} H. Kraft, Geometric methods in representation theory, {\em Lecture Notes in Math.} {\bf 944} (1982), 180--258.

\bibitem{lubmag} A.  Lubotzky and A.R. Magid,  Varieties of representations of finitely generated groups {\em Mem. Amer. Math. Soc.} {\bf 58} (1985), no. 336.

\bibitem{mal}S. Malacarne, Woronowicz's Tannaka-Krein duality and free orthogonal quantum groups, preprint 2016.

\bibitem{ntu}S. Neshveyev and L. Tuset, Compact quantum groups and their representation categories, SMF (2013).

\bibitem{nsp}A. Nica and R. Speicher, Lectures on the combinatorics of free probability, Cambridge Univ. Press (2006).

\bibitem{rwe}S. Raum and M. Weber, The full classification of orthogonal easy quantum groups, {\em Comm. Math. Phys.} {\bf 341} (2016), 751--779.


\bibitem{tw2}P. Tarrago and M. Weber, Unitary easy quantum groups: the free case and the group case, preprint 2015.


\bibitem{wa1}S. Wang, Free products of compact quantum groups, {\em Comm. Math. Phys.} {\bf 167} (1995), 671--692.

\bibitem{wa2}S. Wang, Quantum symmetry groups of finite spaces, {\em Comm. Math. Phys.} {\bf 195} (1998), 195--211.



\bibitem{wo1}S.L. Woronowicz, Compact matrix pseudogroups, {\em Comm. Math. Phys.} {\bf 111} (1987), 613--665.

\bibitem{wo2}S.L. Woronowicz, Tannaka-Krein duality for compact matrix pseudogroups. Twisted SU(N) groups, {\em Invent. Math.} {\bf 93} (1988), 35--76.

\end{thebibliography}
\end{document}